\documentclass[12pt]{amsart}
\usepackage{amsbsy,amsthm,amsfonts,amsmath,amssymb,eucal,amscd,graphics,enumerate, ifthen, indentfirst, stmaryrd,xspace,verbatim, epic, eepic,epsfig,verbatim}
\usepackage[all]{xypic}
\include{el_biblio.bib}

\newcommand{\iX}[1]{{\mathcal I}_{\mu_r}({\mathcal X})}
\def\clX{{\mathcal X}}
\def\clY{{\mathcal Y}}

\def\bix{{\overline{I}_{\mu}(\clX)}}
\def\biy{{\overline{I}_{\mu}(\clY)}}
\def\bix1{{\overline{I}_{\mu_{r_1}}(\clX)}}
\def\biy2{{\overline{I}_{\mu_{r_2}}(\clY)}}

\newtheorem{definition}{Definition}[section]
\newenvironment{defi}{\begin{definition} \rm}{\end{definition}}

\newtheorem{prop}[definition]{Proposition}

\newtheorem{coro}[definition]{Corollary}
\newtheorem{theo}[definition]{Theorem}
\newtheorem{remark}[definition]{Remark}

\newtheorem{remarkdef}[definition]{Remark-Definition}

\newcommand{\ex}[0]{\operatorname{Ext}}

\newcommand{\NS}[0]{\operatorname{NS}}
\newcommand{\T}[0]{\operatorname{T}}
\newcommand{\h}[0]{\operatorname{H}}
\newcommand{\FM}[0]{\operatorname{FM}}
\renewcommand{\dim}[0]{\operatorname{dim}}

\newtheorem{remarks}[definition]{Remarks}

\newtheorem{example}[definition]{Example}

\newtheorem{examples}[definition]{Examples}

\newtheorem{nothing}[definition]{$\!\!$}

\newtheorem{definition*}{Definition}[section]
\newenvironment{defi*}{\begin{definition*} \rm}{\end{definition*}}
\newtheorem{prop*}[definition*]{Proposition}
\newtheorem{lemm*}[definition*]{Lemma}
\newtheorem{coro*}[definition*]{Corollary}
\newtheorem{theo*}[definition*]{Theorem}
\newtheorem{remark*}[definition*]{Remark}
\newenvironment{rema*}{\begin{remark*} \rm}{\end{remark*}}
\newtheorem{remarks*}[definition*]{Remarks}
\newenvironment{remas*}{\begin{remarks*} \rm}{\end{remarks*}}
\newtheorem{example*}[definition*]{Example}
\newenvironment{exam*}{\begin{example*} \rm}{\end{example*}}
\newtheorem{examples*}[definition*]{Examples}\begin{large}                                             \end{large}
\newenvironment{exams*}{\begin{examples*} \rm}{\end{examples*}}

\renewcommand{\tilde}{\widetilde}

\title{On K3 FIBRATIONS: towards Mirror-Symmetry}
\address{ Universitat, Departament de Matematiques, Edifici C, Facultad de Ciencies, 08193 Bellaterra, Barcelona}
\address{Dipartimento di Matematica "Guido Castelnuovo"
Sapienza Universit\`a di Roma
P.le Aldo Moro, 5 - 00185 Roma}
\email{martinez@mat.uniroma1.it}

\author[Cristina Mart{\'\i}nez]{Cristina Mart{\'\i}nez}
\email{cmartine@mat.uab.cat}

\begin{document}
\sloppy
\date{\today}


\begin{abstract}
%SYZ mirror conjecture predicts that a Calabi Yau manifold $X$
%consists of a family of tori which are dual to a family of special
%Lagrangian tori on the mirror dual manifold $\hat{X}$. Here we
%consider a fibration of polarized abelian varieties and we construct
%a dual one. Moreover we prove that they are derived equivalent.
Given $X$ a K3 surface, a mirror dual to $X$ can be identified with a component of the moduli space of semistable sheaves on $X$. We consider  fibrations by $K3$ surfaces over a one dimensional base that are Calabi-Yau and we charaterize the dual fibration that turns to be derived equivalent to the original one relating the problem to mirror symmetry.
\vspace{0.5cm}

% R$\grave{e}$sum\'e.
 %\'Etant donn\'e une surface K3, que l'on d\'enote X, une duale miroir \'a X peut tre identifi\`ee \'a une composante de l'espace de modules de faisceaux semi-stables sur X. %
% on peut identifier une  duale miroir \' a X, \'a un component de l'espace des modules des faisceaux semi-stables sur X. 
% Ensuite, on consid\`ere des fibrations par des surfaces K3 sur des base unidimensionnelle qui sont Calabi-Yau, et on caract\'erise la fibration duale. On d\'emontre que cette derni\`ere est \'equivalente \`a celle d'origine, ce qui a un rapport avec le probl\`eme de la sym\'etrie de miroir.

% ƒtant donnŽe X une surface K3, une duale miroir ˆ X peut tre identifiŽe ˆ une composante de l'espace de modules de faisceaux semi-stables sur X. Nous considŽrons des fibrations en surfaces K3 sur des base de dimension 1, qui sont Calabi-Yau, et nous caractŽrisons la fibration duale qui s'avre tre Žquivalente au sens dŽrivŽ ˆ la fibration originale, tout en reliant le problme ˆ la symŽtrie miroir.
 
%La conjecture du "miroir SYZ" pr$\grave{e}$dit q'une vari\'et\'e
%Calabi-Yau $X$ consiste d'une famille de tores qui sont duals d'une
%famille de tores Lagrangiennes particuliers dans la vari\'et\'e
%miroir dual $\hat{X}$. Nous consid\'erons une fibration de
%vari\'et\'es abeliennes polaris\'ees et nous en construisons une
%dual. En plus, nous montrons qu'ils sont d\'eriv\'ees
%\'equivalentes.

\end{abstract}
\maketitle

\section{Introduction}
%A Calabi-Yau space has two kind of moduli spaces, the moduli space
%of inequivalent complex structures and the moduli space of
%symplectic structures. Mirror Symmetry should consist in the
%identification of the moduli space of complex structures on an
%$n-$dimensional Calabi-Yau manifold $X$ with the moduli space of
%complexified K\"ahler structures on the mirror manifold
%$\widehat{X}$. 
%In the case $X$ is an elliptic curve, the moduli of complex
%structures can be identified with the upper half-plane $\mathbb{H}$
%by
%$$\tau \rightarrow \frac{a\tau+b}{c\tau+d}, \ \ \  A=\left(\begin{array}{ll} a &  b \\
%c & d \end{array}\right) \in PSL(2,\mathbb{Z}).$$ We call it
%$X_{\tau}$, where $\tau$ is the Teichm\"uller parameter. The second
%modulus is the k\"ahler class $[w]\in H^{2}(E,\mathbb{C})$
%parametrised by $t\in \mathbb{H}$, as $\int_{F}w=2\pi i t$.
The classical form of mirror symmetry considers mirror pairs of Calabi-Yau 3-folds $X$ and $\hat{X}$, and the symplectic geometry (Gromov-Witten invariants) of $X$ corresponds to the complex geometry (periods) of $\hat{X}$.

Let $X$ be a complex $K3$ surface %over an algebraically closed field $k$, 
and denote by $\NS(X), \T(X), \tilde{\h}(X,\mathbb{Z})$ the N\'eron-Severi lattice, the trascendental lattice and the Mukai lattice of $X$ respectively. 

Let $f:X\rightarrow C$ be a proper morphism of finite type with integral geometric fibres isomorphic to a polarized $K3$ surface over an algebraic curve $C$ of genus $g$. %that is a three dimensional projective non-singular variety fibred by $K3$ surfaces. 
We prove the existence of a projective relative moduli space for stable sheaves on the fibers of $f$ that turns to be derived equivalent to the original fibration and can be considered as a dual fibration. 

Consider the fine moduli space $\mathcal{M}(r,e,s)$ parametrizing $e-$stable %(with respect to the polarization $\mathcal{O}_{X}(1)$)
sheaves $E$ on $X$ such that %definir e-stable R. Thomas
$c_{0}(E)=rk(E)=r$, $c_ {1}(E)=e$ and $\chi(E)=r+s$. Here stability means Gieseker stability as considered in \cite{Sim}. The vector $v=(r,e,s)\in \tilde{\h}(X,\mathbf{Z})=\h^{0}(X,\mathbb{Z})\oplus \h^{1,1}(X,\mathbb{Z})\oplus \h^{4}(X,\mathbb{Z})$ is a class in the topological $K-$theory $\mathcal{K}_{top}(X)$ of the surface and it is called Mukai vector. Our main result is:

\begin{theo}\label{theo1}
 Given a non singular fibration $p:X\rightarrow C$ by $K3$ surfaces 
 with a polarization class $H$ of degree $d$,
 there exists a dual fibration which is derived equivalent to the original one and corresponds to a connected component of the relative
 moduli space $\mathcal{M}^{l}(X/C)$.  %there is also fiberwise equivalence
\end{theo}
The moduli space $\mathcal{M}^{l}(X)$ of semistable sheaves on $X$
with respect to a fixed polarization $l$, in general has infinitely
many components, each of which is a quasi-projective scheme which
may be compactified by adding equivalence classes of semistable
sheaves. An irreducible component $Y\subset \mathcal{M}^{l}(X/C)$ is
said to be fine if $Y$ is projective and there exists a universal
family of stable sheaves, that is, and object of $D^{b}(X\times Y)$
inducing a derived equivalence.

\begin{coro}\label{finecomp}
There exists at least one  fine component of the relative moduli space or equivalently a sheaf $\mathcal{E}$
on a non singular fiber with fixed Mukai vector. %explain what it means fixed Mukai vector=fixed invariants: degree...
\end{coro}

%\begin{theo}\label{theo1}
 %Given a non singular fibration $p:X\rightarrow C$ by $K3$ surfaces 
 %with a polarization class $H$ of degree $d$,
 %there exists a dual fibration which is derived equivalent to the original one and corresponds to a connected component of the relative
 %moduli space $\mathcal{M}^{l}(X/C)$.  %there is also fiberwise equivalence
%\end{theo}

%Here we will study Calabi-Yau spaces that are fibred over the same
%base $B$ by polarized abelian varieties. If $X/B$ is an abelian
%fibration with a global polarization, and we call $X^{\vee}/B$ its
%dual fibration, our main result is:

%\begin{theo} \label{T1}The derived categories of both fibrations $X/B$ and $X^{\vee}/B$ are equivalent.
%\end{theo}

%\begin{coro} There is an equivalence $\phi_{b}:D^{b}(X_{b})\rightarrow
%D^{b}(\widehat{X}_{b})$ for every closed point $b\in B$.
%\end{coro}

\section{Derived categories of split-type Calabi-Yau manifolds}
There are two main mathematical conjectures in Mirror Symmetry,
Kontsevich homological mirror symmetry conjecture and the conjecture
of Strominger, Yau and Zaslow, which predicts the structure of a CY
manifold and how to get the mirror of a given CY manifold. %We first
Recall that by a Calabi-Yau manifold we mean a compact K\"ahler manifold $X$ with trivial canonical bundle $K_{X}$. Many examples of Calabi-Yau manifolds can be constructed by considering fibrations of lower dimensional varieties, that is, elliptic or $K3$ fibrations. These are the so called split-type Calabi-Yau manifolds.

%\begin{defi} A Calabi-Yau manifold $X$ of dimension $n$ is a smooth compact connected $n-$fold with vanishing first Betti number and trivial canonical class $$\Lambda^{n}\Omega_{X}\cong K_{X}\cong \mathcal{O}_{X}.$$
%\end{defi}
%Ejemplo: the Fermat quintic

%Definici—n of Calabi-Yau manifolds that are derived equivalent
A $K3$ surface is a compact complex surface $X$ which is connected and simply connected and has trivial canonical bundle $K_{X}$, i.e., $X$ has a unique (up to constant) nowhere vanishing holomorphic 2-form $w_{X}$. The notion of $K3$ surface is invariant under deformation, i.e., any deformation of a $K3$ surface is a $K3$ surface. Moreover any two $K3$ surfaces are deformations of each other. Hence the lattice $\h^{2}(X;\mathbb{Z})$ with the cup bilinear form 
$$\langle , \rangle:\ \h^{2}(X,\mathbb{Z})\times \h^{2}(X,\mathbb{Z})\rightarrow \mathbb{Z}, \ \ even\ for \ all \
\alpha\ \in \h^{2}(X,\mathbb{Z}),$$
is the same for all $K3$ surfaces $X$ and can be called the $K3$ lattice. Let $e\in \h^{1,1}(X,\mathbf{C})\cap \h^ {2}(X,\mathbf{Z})$ be the class of an ample divisor. Then $(X,e)$ is a polarized $K3$ surface.
The degree of the polarization is an integer $2d$, such that the scalar product $<e,e>=2d=2rs$ where $d, r, s$
are any positive integers and their greatest common divisor $(r,s)$ is 1.

\begin{defi} Two K3 surfaces $X, Y$ are said to be FM partners, if there is an equivalence  $D(X)\cong D(Y)$ of their bounded derived categories of coherent sheaves. The set of isomorphism classes of FM partners of $X$ is denoted by $\FM(X)$.
\end{defi}

%proper map %over an algebraic curve $C$ of genus $g$, 
%when all the fibers are equidimensional, we say that it is a fibration. %We don't impose further assumptions on the base.
%It is of interest from the point of view of Mirror Symmetry, the case in which the total fibration is Calabi-Yau. In this case, Strominger, Yau and Zaslow have conjectured how would it be the structure of the mirror fibration. Mirror dual Calabi-Yau manifolds should be fibred over the same base in such a way that generic fibres are dual tori, and each fibre of any of these two fibration is a Lagrangian submanifold.

%\begin{defi} Let $(X,w)$ be a holomorphic symplectic manifold (not necessarily compact) of dimension $2r$. A Lagrangian fibration is a proper map $h:X\rightarrow B$ onto a manifold $B$ such that the general fibre $F$ of $h$ is Lagrangian, that is, $F$ is connected, of dimension $r$, and the restriction $w|_{F}=0$ vanishes. This implies that the smooth fibres of $h$ are complex tori.
%\end{defi}

\subsubsection*{The homological mirror symmetry  Conjecture.}
Homological mirror symmetry conjecture due to Kontsevich,  asserts that there should be an equivalence of categories behind mirror
duality, one category being the derived category of coherent sheaves on a Calabi-Yau
manifold $D(X)$ and the other one being the Fukaya category $DFuk(\hat{X})$ of the mirror Calabi-Yau manifold.

Let $\pi: Y\rightarrow S$ be a fibration by $K3$ surfaces with a relative polarization.%a section. 
This means that on $Y$ we have a polarization class $H$ such that its restriction to each fibre  
%restricted to the fiber 
$H|_{X_ {t}}=e$ is the polarization class of the corresponding fibre.
We can assume that the fibration is Calabi-Yau. Since  the singular fibers are normal crossing divisors,  and the total space and the base are projective varieties the fibration morphism is automatically proper. We are assuming that the fibers are equidimensional and therefore the morphism is flat.  By the theorem of U. Person and H. Pinkham \cite{PP}, there exists a birational map $\varphi: X\rightarrow X'$  where $X'$ has trivial canonical bundle and it is an isomorphism over the smooth locus such that the following diagram is commutative:

 $$ \xymatrix{ X  \ar[rr]^{\varphi}  \ar[rd]^{\pi} &  & X'\ar[ld] _{\pi '}    \\
&  B & } $$

Now, by Bridgeland theorem (see \cite{Bri}), two birational  3-folds have equivalent derived categories.

%Let  $p:X\rightarrow B$ be a fibration by  abelian varieties with a relative polarization or ample line bundle defined over the total fibration such that the restriction to each fibre is the polarization class on the corresponding fibre. The existence of a relative polarization for the fibration does not necessarily  imply the existence of a section.
%\noindent In the fibration $X/B$ singular fibres can appear. In this case, except for some particular cases, we don't know how does the dual abelian variety look like. The idea is to replace the abelian variety by one that is derived equivalent to it.
We consider the moduli problem of the dual fibration, that is, the dual fibration as the stack representing the Picard functor, that is, the moduli functor of semistable
sheaves on the fibres that contains line bundles of degree 0 on
smooth fibres. The corresponding coarse moduli space is not a fine
moduli space due to the presence of singular fibres. Let us call
$Y^{\vee}$ the dual fibration when it exists and satisfying the
property  that over the smooth locus the fibres correspond to the
dual $K3$ surfaces of the original fibration.

\section{Proof of Theorem 1.1}

%\begin{con}
%\subsubsection*{Conjecture}
%Two Calabi-Yau threefolds $C_{1}, C_{2}$ that are
%fibred over the same base $\mathbb{P}^{1}$ in such a way that the
%fibres are abelian surfaces and derived equivalent are derived
%equivalent themselves. This prediction is according with SYZ mirror
%symmetry conjecture.
%\end{con}

%First we fix our attention on the smooth locus.
 Let $\Sigma(p) \hookrightarrow C$ be the discriminant locus of
$p$, that is, the closed subvariety %in the parameter space $B$
corresponding to the singular fibres.  From Hironaka's theorem on the resolution of singularities, we may assume that the singular fibers are normal crossing divisors. Thus the fibration morphism is automatically proper and flat.

For every $t \in C-\Sigma(p) $, consider the $K3$ surface $X_{t}$ and its corresponding Mukai vector $(r_{t},e, s_{t})$, where  $2r_{t}s_{t}=(H_{t})^{2}=H^{2}=2d$.
 %By Theorem \ref{MO}, 
By Mukai's Theorem (see \cite{Muk}), we may associate to $X_{t}$ a 2-dimensional moduli space $\mathcal{M}(r_{t},e,s_{t})$ which is
 % we single out a  dual $K3$ surface $\hat{X}_{t}$ of $X_{t}$ 
 %In particular, choosing coordinates $r=2,\ l=0,\ s=4$ for the Mukai
%vector, we obtain a 2-dimensional symplectic manifold, that is,  a
a $K3$ surface as well with the same derived category to $X$, thus it is % We will denote it as $\widehat{X}$.
 a FM partner. %that is, it has the same derived category. % This is always possible since $|FM(X_{t})|$ is finite. see \cite{HLOY} 
 We observe that although the degree of the polarization is constant in $t$, the rank of the fibres can jump for some $t \in C$. However the condition of the Picard rank being one is open in the Zariski topology and it determines an open set
 $$C^{1}:=\{t\in C|\,\, \NS(X_{t})=\mathbb{Z}H_{t}\}.$$

Now if $s\in C^{1}$, then $H|_{X_{t}}=H_{t}=l$ is an ample divisor and since the number of Mukai partners depends on the prime decomposition
 $l^{2}=2d=2p_{1}^{e_{1}}\ldots p^{e_{m}}_{m}$, where $k\geq 0$, $e_{i}\geq 1$ and $p_{i}$ primes with $p_{i}\neq p_{j}$, if $i\neq j$,  there is a description of the FM partners of the $K3$ surface  in terms of the Mukai vectors of the moduli spaces associated (see \cite{St}). We need to single out a unique Mukai dual $K3$ surface. For example, the reflected Mukai vectors $(r_{t},e, s_{t})$ and $(s_{t},e,r_{t})$ give isomorphic moduli spaces $\mathcal{M}(r_{t},e,s_{t})\cong \mathcal{M}(s_{t},e,r_{t})$ even if the original $K3$ surfaces are not isomorphic. Thus, this choice gives rise to  different dual fibrations.
 
 If the rank of the Neron Severi group $\NS(X_{t})$ is bigger than 12, according to Morrison \cite{Mo},  there exists a torsion free semistable bundle on $X_{t}$, and the choice of dual $K3$ surface is unique in this case.
 
 Consider  the product  $X_ {t} \times \hat{X}_{t}$  of the corresponding $K3$ surface $X_{t}$ with its Mukai dual.
 Then we consider the universal family $\mathcal{P}_{t}$ over the product $X_{t}\times \hat{X}_{t}$. %referencia de Nikulin
Proceeding as in Proposition 2.4 of \cite{Ma}, extending the family
$\mathcal{P}:=\{\mathcal{P}_{t}: t \in B\}$ over the non singular
locus by Deligne theorem (\cite{Del}), the class of the polarization
is invariant by the action of the monodromy group of the singular
fibres, thus $\mathcal{P}$ extends to an object $\mathcal{F}$ over
the whole fibration.

The family does not need to be universal, but according to
C\v{a}ld\v{a}raru (see \cite{Cal}), a quasi-universal or twisted
universal family sheaf always exists and thus the dual fibration
$(X/C)^{\vee}$  is the coarse moduli space induced by $\mathcal{F}$.
The fibration constructed thus far,  is a connected component $M$ of
the relative moduli space $\mathcal{M}^{l}(X/C)$ of stable sheaves
on $p$ with respect to the polarization, (Prop. 3.4. of \cite{BM}).
There exists a unique $\alpha$ in the Brauer group $Br(M)$ of $M$
with the property that an $p^{*}_{M}\alpha^{-1}$ twisted universal
sheaf exists on $X\times M$, where $p_{M}$ is the projection map
from $X\times M$ to M, and it is the obstruction to the existence of
a universal sheaf on $X\times M$. This twisted universal sheaf
yields an equivalence (Theorem 1.2 of \cite{Cal}).
$$D^{b}(M,\alpha)\cong D^{b}(X).$$
So both fibrations are derived equivalent.
$\square$
\section{Remarks and conclusions}
There exists at least one  fine component of the relative moduli space or equivalently a sheaf $\mathcal{E}$
on a non singular fiber with fixed Mukai vector. %explain what it means fixed Mukai vector=fixed invariants: degree...
%\end{coro}
%{\bf Claim}.
A closed point of a relative moduli space corresponds to a sheaf
$\mathcal{E}$ on a fibre (not to a sheaf on the whole fibration).
Let $X_{s}$ be a $K3$ surface or an abelian surface. The tangent
space at that point to the moduli space of sheaves $M(X/S)$ on the
fibration, can be identified with
$$T_{M}(\mathcal{E})\cong \ex^{1}_{S}(\mathcal{E},\mathcal{E}).$$

If $\ex^{2}_{S}(\mathcal{E},\mathcal{E})=0$, then $M$ is smooth at
$\mathcal{E}$. There are bounds (Corollary 4.5.2 of
\cite{HL}), $$\ex^{1}(\mathcal{E},\mathcal{E})\geq \dim_{[\mathcal{E}]}M \geq
\ex^{1}(\mathcal{E},\mathcal{E})-\ex^{2}(\mathcal{E},\mathcal{E}).$$ %page 101

In general to construct such components $Y$ of the relative moduli space,
we assume that there exists a divisor $L$ on $X$ and integer numbers $r, s>0$,
such that there exists a sheaf $\mathcal{E}$ on a non singular fiber $X_ {t}$
which is stable with respect to $H_ {t}$ and $s=ch_ {2}(\mathcal{E})+r$.
% there exists such a sheaf \mathcal{E}
The component $Y(\mathcal{E})$ containing the class of the sheaf $\mathcal{E}$ is a fine projective moduli space
and the fibration $q:Y\rightarrow B$ is equidimensional. %according to AHS
Thus there is a universal family on the product $Y \times
Y(\mathcal{E})$ that gives the equivalence of the derived categories
of both fibrations over $B$.

\begin{prop} \label{prop1} Every fine projective component $Y$
of the relative moduli space $\mathcal{M}^{e}(X/B)$ of stable
sheaves with respect to a fixed polarization $e$ is derived
equivalent to the original Calabi-Yau fibration $(X/B)$ and
therefore are derived equivalent between them. Conversely, any
projective variety derived equivalent to the original fibration is a
component of the relative moduli space.
%are all possible equivalences of $D^{b}(X)$.
\end{prop}
\begin{proof}[Proof.]
By Corollary \ref{finecomp} we can consider  components $Y$ of the relative moduli space $\mathcal{M}^{e}(X/B)$ of stable
sheaves %of fixed numerical invariants $r, s$
on the fibers of  the CY fibration $(X/B)$,  stable with respect to
the polarization $e$.  It is a fine moduli space, so there is a
universal sheaf $\mathcal{P}$ over the product $X\times Y$.
Bridgeland and Maciocia proved in \cite{BM} that $Y$ is a
non-singular projective variety, $\widehat{p}:Y\rightarrow B$ is a
$K3$ fibration and the integral functor $D^{b}(Y)\rightarrow
D^{b}(X)$ with kernel $\mathcal{P}$  is an equivalence of derived
categories, that is, a Fourier-Mukai transform. It is Calabi-Yau
because one has $D^{b}(X)\cong D^{b}(Y)$.

Now, we start with an equivalence $D^{b}(Y)\cong D^{b}(X)$, then by
a result of Orlov \cite{Or1}, it is given by an object
$\mathcal{E}\in D^{b}(X\times_{B}Y)$ which satisfies a Calabi-Yau
condition and thus by Theorem \ref{theo1} this defines a fine
component of the relative moduli space. All the equivalences of the
original fibration are obtained in this way.
\end{proof}

%\begin{coro} 
%There is an equivalence $\phi_{b}:D^{b}(X_{b})\rightarrow D^{b}(\widehat{X}_{b})$ for every closed point $b\in B$.
%\end{coro}
%\begin{proof}
%By Theorem \ref{T1}, the integral functor $\phi^{X\rightarrow \widehat{X}}_{\mathcal{E}}: D^{b}(X)\rightarrow D^{b}(\widehat{X})$ is an equivalence of derived categories, where
%$\widehat{\rho}: \widehat{X}\rightarrow B$ is the dual abelian
%fibration. It follows from Prop. 2.15 of  \cite{HLS} that there is fibrewise equivalence
%$\phi_{b}:D^{b}(X_{b})\rightarrow D^{b}(\widehat{X}_{b})$.% Reciprocally if there is a global derived equivalence between both fibrations, restricted to fibres is a  equivalence of  derived categories.
%\end{proof}
%\section{Twisted fibrations}


\begin{thebibliography}{AAA9}
\small
\bibitem[BM]{BM} T. Bridgeland and A. Maciocia, {\it Fourier-Mukai transforms for $K3$ and elliptic fibrations}, J. Algebraic Geometry 11 (2002), no. 4, 629-657.
\bibitem[Bri]{Bri} T. Bridgeland, {\it Flops and derived categories}, math.AG/0009053
\bibitem[Cal]{Cal} A. C\u{a}ld\u{a}raru, {Non-fine moduli spaces of sheaves on K3 surfaces}, Internat. Math. Res. Notices 2002 (2002), no. 20, 1027-1056.
\bibitem[Del]{Del} P. Deligne, {\it Les intersections compl{\`e}tes de
nieveau de Hodge un}, Inventiones math. 15, 237-250 (1972).
%\bibitem[Dijk]{Dijk} R. Dijkgraaf,  {\it Mirror symmetry and elliptic curves}, The moduli space of curves (Texel Island, 1994), 149-163, Progr. Math., 129, Birkhäuser Boston, MA, 1995.
%\bibitem[DP]{DP}R. Donagi and T. Pantev, {\it Torus fibrations, gerbes, and duality},
%math.AG/0306213, in Memoirs of the AMS.
%\bibitem[HLS]{HLS} D. Hen{\'a}ndez Ruip{\'e}rez, A. C. L{\'o}pez
%Mart{\'\i}n and F. Sancho de Salas, {\it Relative integral functors for
%singular fibrations and singular partners}, to appear in: J. Eur.
%Math. Soc., math.AG/0610319.
\bibitem[HL]{HL} D. Huybrechts, Manfred Lehn, {\it The geometry of moduli spaces of Sheaves}.
A Publication of the Max-Planck-Institut f\"ur Mathematik, Bonn.

\bibitem[Ma]{Ma} C. Mart\'inez, {\it Abelian fibrations and SYZ mirror conjecture"}, C. R. Acad. Sci.
Paris. Ser. I, Vol. 350, no. 13, 689-692 (2012).
\bibitem[Mo]{Mo} D. Morrison, {\it On K3 surfaces with large Picard number}, Invent. Math. 75 (1984) 105-121.
\bibitem[Muk]{Muk} S. Mukai, {\it Symplectic structure of the
moduli space of sheaves on an abelian or $K3$ surface}, Invent.
math. {\bf 77} (1984), 101-116.


%\bibitem[MT]{MT} C. Mart\'inez, A. N. Todorov, {\it The derived categories of Calabi-Yau fibrations},  arXiv:0901.0509.
\bibitem[Mu]{Mu} S. Mukai, {\it Duality between $D(X)$ and $D(\hat{X})$ with its application to Picard sheaves}, Nagoya Math. J. Vol. 81 (1981), 153-175.
\bibitem[Or1]{Or1} D. Orlov. {\it Equivalences of
derived categories and $K3$ surfaces}. Algebraic geometry, 7.  J.
Math. Sci. (New York) 84 (1997),  no. 5, 1361--1381.
\bibitem[PP]{PP} U. Persson, H. Pinkham, {\it Degeneration of surfaces with trivial canonical bundle}, Annals of Mathematics, 113 (1981), 45-66.
\bibitem[Sim]{Sim} C. T. Simpson. Moduli of representations of the fundamental group of a smooth projective variety I. Inst. Hautes \'Etudes Sci. Publ. Math., 79 (1994), pp. 47-129.
\bibitem[St]{St} P. Stellari, {\it Some remarks about the FM-partners of $K3$ surfaces with Picard rank 1 and 2}, Geom. Dedicata 108,  (2004), 1-13.

\end{thebibliography}
\end{document}